\documentclass[11pt]{article}

\usepackage{amsmath, amsthm, amssymb}
\usepackage{enumitem}
\usepackage{mathtools}
\usepackage{etoolbox}
\usepackage{xcolor}
\usepackage[noadjust]{cite}
\usepackage{authblk}
\usepackage{nccmath}
\usepackage[normalem]{ulem}

\theoremstyle{plain}
\newtheorem{theorem}{Theorem}[section]
\newtheorem{definition}[theorem]{Definition}

\newtheorem{proposition}[theorem]{Proposition}
\newtheorem{corollary}[theorem]{Corollary}

\theoremstyle{definition}
\newtheorem{example}[theorem]{Example}
\newtheorem*{remark}{Remark}

\setlist{noitemsep, topsep=1ex, parsep=1ex, partopsep=1ex}

\allowdisplaybreaks
\appto\normalsize{
	\abovedisplayskip=2ex plus 1ex minus 1ex
	\belowdisplayskip=2ex plus 1ex minus 1ex
	\abovedisplayshortskip=2ex plus 1ex minus 1ex
	\belowdisplayshortskip=2ex plus 1ex minus 1ex}
\appto\small{
	\abovedisplayskip=2ex plus 1ex minus 1ex
	\belowdisplayskip=2ex plus 1ex minus 1ex
	\abovedisplayshortskip=2ex plus 1ex minus 1ex
	\belowdisplayshortskip=2ex plus 1ex minus 1ex}

\AtBeginDocument{%
  \mathchardef\mathcomma\mathcode`\,
  \mathcode`\,="8000 
}
{\catcode`,=\active
  \gdef,{\mathcomma\discretionary{}{}{}}
}

\newcommand{\gap}{\vspace{1ex}}
\newcommand{\diag}{\operatorname{Diag}}
\newcommand{\tr}{\operatorname{tr}}
\newcommand{\trace}{\operatorname{trace}}
\newcommand{\Aut}{\operatorname{Aut}}
\newcommand{\rank}{\operatorname{rank}}
\newcommand{\R}{\mathcal{R}}
\newcommand{\Rn}{\mathcal{R}^n}
\newcommand{\Sn}{\mathcal{S}^n}
\newcommand{\Hn}{\mathcal{H}^n}
\newcommand{\V}{\mathcal{V}}
\newcommand{\W}{\mathcal{W}}
\newcommand{\one}{\mathbf{1}}
\newcommand{\lplus}{\Lambda_{+}}
\newcommand{\bfA}{\mathbf{A}}
\newcommand{\bfB}{\mathbf{B}}

\newcommand{\norm}[1]{\left\Vert #1 \right\Vert}
\newcommand{\ip}[2]{\left< #1, #2 \right>}
\newcommand{\set}[2]{\left\lbrace #1 : #2 \right\rbrace}

\title{Minimal polynomials, scaled Jordan frames, and Schur-type majorization  in hyperbolic systems}

\author[1]{M. Seetharama Gowda}
\author[2]{Juyoung Jeong} 
\author[3]{Sudheer Shukla}

\affil[1]{\small Department of Mathematics and Statistics\\University of Maryland Baltimore County\\Baltimore, Maryland 21250, United States}
\affil[2]{\small Department of Mathematics\\Soongsil University\\Seoul 06978, South Korea}
\affil[3]{\small Department of Mathematics\\University of Maryland, College Park, Maryland 20742, United States}

\date{\today}

\begin{document}

\maketitle

\begin{abstract}
    Corresponding to a hyperbolic system $(\V, p, e)$, where $\V$ is a real finite-dimensional vector space and $p$ is a hyperbolic polynomial of degree $n$ in the direction $e$ \cite{garding1959inequality, bauschke2001hyperbolic}, we consider the eigenvalue map $\lambda: \V \to \Rn$ and the hyperbolicity cone $\lplus$. In such a system, a scaled Jordan frame is defined as a finite set of rank-one elements whose sum lies in the interior of $\lplus$. We show that when the system has a scaled Jordan frame and $n\geq 2$, $p$ and its derivative polynomial $p^\prime$ are minimal polynomials (generating their respective hyperbolicity cones), thereby extending a result of  Ito and Louren{\c c}o \cite{ito2023automorphisms} proved in the setting of a rank-one generated (proper) hyperbolicity cone.  When each element of a scaled Jordan frame has trace one and the total sum is $e$  (such a set is called a Jordan frame), we show that the frame is orthonormal relative to the semi-inner product induced by $\lambda$ with exactly $n$ elements, and $\V$ contains a copy of $\Rn$ (as a Euclidean Jordan algebra). We also present a Schur-type majorization result corresponding to a Jordan frame and an $e$-doubly stochastic $n$-tuple.
\end{abstract}

\vspace{2ex}



\section{Introduction} 

Hyperbolic polynomials/systems appear in various branches of mathematics, including optimization, algebraic geometry, combinatorics, etc.; see \cite{garding1959inequality, bauschke2001hyperbolic, renegar2006hyperbolic, ito2023automorphisms, gurvits2004combinatorics}.  Motivated by the work of Ito and Louren{\c c}o \cite{ito2023automorphisms}, in this article, we introduce the concept of a scaled Jordan frame and describe results that address the minimality of the underlying polynomial and Schur-type majorization. We begin with a brief review.

\gap
Consider a finite-dimensional real vector space $\V$ and a nonzero element $e \in \V$. A homogeneous polynomial $p$ of degree $n$ ($\geq 1$) on $\V$ is said to be \emph{hyperbolic in the direction of $e$} (or relative to $e$) \cite{garding1959inequality} if $p(e) \neq 0$ and for each $x \in \V$, the univariate polynomial $t \mapsto p(te-x)$ has (only) real zeros. The triple $(\V, p, e)$ will then be called a \emph{hyperbolic system of degree} $n$. Now, in a hyperbolic system $(\V, p, e)$, we consider $x\in \V$ and arrange the $n$ roots of $p(te-x) = 0$ in decreasing order to create the vector $\lambda(x) \coloneq \big (\lambda_1(x),\lambda_2(x),\ldots, \lambda_n(x)\big)^T$ in $\Rn$, the entries of which are called the \emph{eigenvalues} of $x$. This \emph{eigenvalue map} $\lambda:\V \to \Rn$  induces a semi-inner product on $\V$  \cite{bauschke2001hyperbolic}; see Proposition \ref{bauschke2001hyperbolic} below. In this setting, the \emph{hyperbolicity cone} $\lplus$ and the \emph{derivative polynomial} $p^\prime$ \cite{bauschke2001hyperbolic, renegar2006hyperbolic} are defined by 
\begin{equation} \label{derivative polynomial}
	\lplus \coloneq \set{x \in \V}{\lambda(x) \geq 0} \quad \text{and} \quad p^\prime(x) \coloneq \frac{d}{dt}\,p(te+x) \Big\vert_{t=0}. 
\end{equation}

Given a hyperbolic system $(\V,p,e)$ and its hyperbolicity cone $\lplus$, a homogeneous polynomial $q$ on $\V$ is said to be minimal if $q$ has the lowest degree among all hyperbolic polynomials on $\V$ that induce $\lplus$ in direction $e$. An important result due to Helton and Vinnikov \cite[Lemma 2.1]{helton2007lmi}  says that such a minimal polynomial exists, unique up to a positive constant, and divides any polynomial that induces $\lplus.$ Given a hyperbolic system $(\V,p,e)$, the question of when $p$ is minimal becomes interesting. In a recent paper, Ito and Louren{\c c}o \cite{ito2023automorphisms} show that $p$ and $p^\prime$ are minimal when $\lplus$ is an ROG-cone (a pointed cone where every extreme direction of $\lplus$ is generated by a rank-one element, i.e., an element $x$ with $\lambda_1(x)>0$, $\lambda_2(x)=\lambda_3(x)=\cdots=\lambda_n(x)=0$).  In our paper, we weaken the ROG-cone condition by assuming the existence of a finite set of rank-one elements whose sum lies in the interior of $\lplus$. We call such a set a  `\emph{scaled Jordan frame}' and say that $(\V,p,e)$ has the `scaled Jordan frame' property. It turns out that the scaled Jordan frame property is `hereditary' (i.e., it is inherited by the derivative polynomial $p^\prime$), while for $n\geq 4$, the ROG-cone property is not.

\gap
A special case of a scaled Jordan frame in $(\V,p,e)$ is a `\emph{Jordan frame}' where each element (called a `primitive idempotent') has $1$ as the only nonzero eigenvalue, and the sum of all elements is $e$. The coinage of the term `Jordan frame' is not accidental: we show that such a finite set of primitive idempotents is orthonormal relative to the semi-inner product induced by $\lambda$ and contains exactly $n$ elements, which is the defining property of a Jordan frame in a Euclidean Jordan algebra of rank $n$. 

\gap
To motivate our Schur-type majorization results (and our discussion of primitive idempotents) in the setting of hyperbolic systems, we recall some concepts and results from the topic of Euclidean Jordan algebras. Consider a Euclidean Jordan algebra $(\V, \ip{\cdot}{\cdot}, \circ)$ of rank $n$ with unit $e$ \cite{faraut1994analysis}. In this setting, a nonzero element $c \in \V$ is said to be an idempotent if $c \circ c = c$ and a primitive idempotent if it is an idempotent but not a sum of two idempotents. A Jordan frame in $\V$ is a finite set of mutually orthogonal primitive idempotents with sum $e$. By the spectral decomposition theorem \cite{faraut1994analysis}, every element $x\in \V$ has a decomposition of the form $x = x_1e_1 + x_2e_2 + \cdots + x_ne_n$, where $\{e_1, e_2, \ldots, e_n\}$ is a Jordan frame and $x_1, x_2, \ldots, x_n$ are real numbers (called \emph{eigenvalues} of $x$) with $x_1 \geq x_2 \geq \cdots \geq x_n$. The vector $\lambda(x) \coloneq (x_1, x_2, \ldots, x_n)^T$ in $\Rn$ is called the \emph{eigenvalue vector} of $x$. Then, the eigenvalue vector of an idempotent is of the form $(1,1,\ldots, 1,0,\ldots, 0)^T$ while that of a primitive idempotent is of the form $(1,0,\ldots, 0)^T$. Using the eigenvalue map $\lambda:x\mapsto \lambda(x)$, one defines the symmetric cone and the determinant:
\[ \V_+ \coloneq \set{x \in \V}{\lambda(x) \geq 0} \quad \text{and} \quad \det(x) \coloneq x_1 x_2 \cdots x_n. \]
It is known that $(\V, p, e)$, where $p(x) = \det(x)$, becomes a hyperbolic system with $\V_+$ as its hyperbolicity cone. Now, an analog of Schur's majorization theorem for  Euclidean Jordan algebras asserts the following: Let $\{c_1, c_2, \ldots, c_n\}$ be a Jordan frame. For any $x \in \V$, let 
\begin{equation} \label{S map}
	\diag(x) \coloneq \ip{x}{c_1} c_1 + \ip{x}{c_2} c_2 + \cdots + \ip{x}{c_n} c_n
\end{equation}
(which is the `diagonal' of $x$ in the Peirce decomposition of $x$ relative to the given Jordan frame). Then, $ \lambda(\diag(x)) \prec \lambda(x)$ in $\Rn$. In the classical case, this reduces to Schur's result, which says that for a Hermitian matrix, the diagonal is majorized by the eigenvalue vector. (To see this, one observes that on the Euclidean Jordan algebra $\Hn$ of all $n \times n$ Hermitian matrices, corresponding to the canonical Jordan frame $\{E_1, E_2, \ldots, E_n\}$, the transformation $\diag(X) = \sum_{k=1}^{n} \ip{X}{E_k} E_k$ is the `diagonal' of $X$.)
This result turns out to be a particular instance of a broader result \cite[Theorem 6]{gowda2017positive}: Suppose $T$ is a  doubly stochastic linear transformation on the Euclidean Jordan algebra $\V$, which, by definition, satisfies the conditions 
\begin{equation} \label{defn of ds transformation}
	T(\V_+) \subseteq \V_+, \quad T(e) = e, \quad \text{and} \quad T^*(e) = e,
\end{equation}
where $T^*$ denotes the adjoint of $T$. Then, $\lambda(T(x)) \prec \lambda(x)$ for all $x \in \V$.  This broader result on $T$ yields much more than the analog of Schur's result mentioned above. Let $\{c_1, c_2, \ldots, c_n\}$ be a Jordan frame in $\V$ and $\bfA \coloneq \big[ a_1, a_2, \ldots, a_n \big]$ be an $n$-tuple of elements in $\V$ such that $a_i \in \V_+$, $\tr(a_i) = 1$ for all $i$, and $\sum_{i=1}^{n} a_i = e$.
Then, each of the following transformations is doubly stochastic and so has the majorization property $\lambda(T(x)) \prec \lambda(x)$:
\begin{equation} \label{trio}
	T_1(x) \coloneq \sum_{i=1}^{n} \ip{x}{a_i} c_i, \;\;
	T_2(x) \coloneq \sum_{i=1}^{n} \ip{x}{c_i} a_i, \;\; \text{and} \;\; 
	T_3(x) \coloneq \sum_{i=1}^{n} \ip{x}{a_i} a_i.
\end{equation}
Our objective in this paper is to go beyond Euclidean Jordan algebras and determine whether these concepts and results extend to the broader context of hyperbolic systems.

\gap
Here is a summary of our contributions. In the setting of a hyperbolic system $(\V, p, e)$, where $p$ has degree $n$, we introduce the concepts of a primitive element, a scaled Jordan frame, and an $e$-doubly stochastic $n$-tuple. We show that

\begin{itemize}
	\item the scaled Jordan frame property is weaker than the ROG-cone property;
	\item the scaled Jordan frame property is hereditary, while for $n\geq 4$, the ROG-cone property is not;
	\item when $n\geq 2$ and $(\V,p,e)$   has the scaled Jordan frame property, $p$ and its derivative polynomial $p^\prime$ are minimal polynomials (generating their respective hyperbolicity cones);
	\item a Jordan frame is orthonormal (relative to the semi-inner product induced by the hyperbolic structure) and contains $n$ elements;
	\item the existence of a Jordan frame implies that $\V$ contains a copy of $\Rn$ (as a Euclidean Jordan algebra);
	\item there is a Schur-type majorization result corresponding to a Jordan frame and an $e$-doubly stochastic $n$-tuple. 
\end{itemize}

\section{Preliminaries}

The Euclidean space $\Rn$ (in which elements are viewed as column vectors) always carries the usual inner product. Here, we write (the boldface) $\one$ for the vector of ones. For any index $k$, $1 \leq k \leq n$, we write $\one_k$ for the vector with exactly $k$ ones followed by $n-k$ zeros. For $u \in \Rn$, $u^\downarrow$ denotes its decreasing rearrangement vector. In $\Rn$, the (classical) majorization \cite{marshall2011inequalities} is defined as follows: Given two vectors $u, v \in \Rn$, we say that \emph{$u$ is majorized by $v$}  in $\Rn$ and write $u \prec v$ if for all $1 \leq k \leq n$, 
\[ \sum_{i=1}^{k} u_i^\downarrow \leq \sum_{i=1}^{k} v_i^\downarrow \quad \text{and} \quad \sum_{i=1}^{n} u_i^\downarrow = \sum_{i=1}^{n} v_i^\downarrow. \]
We recall a basic result \cite[Theorem II.1.10]{bhatia1997matrix}: \emph{$u \prec v$ if and only if $u$ is in the convex hull of all vectors obtained by permuting the entries of $v$.} \\ The definition of a hyperbolic system $(\V,p,e)$ was given in the Introduction. In this system, sometimes, to show the dependence of the eigenvalue map $\lambda$ on the pair $(p,e)$, we write $\lambda_{(p,e)}$ (and when $e$ is fixed, just $\lambda_p$). \emph{Henceforth, unless mentioned otherwise, we assume that a generic hyperbolic system $(\V,p,e)$ has degree $n$,} where 
\begin{center}
	\fbox{\begin{minipage}[c][6mm][c]{32mm}
			\centering
			$n = \deg(p)$.
	\end{minipage}}
\end{center}

\begin{definition}\textup{\cite{bauschke2001hyperbolic}}
	Let $(\V, p, e)$ be a hyperbolic system. We say that $p$ is \emph{complete} if 
	\[ \lambda(x) = 0 \implies x=0. \] 
	In this case, we say that $(\V, p, e)$ is a \emph{complete hyperbolic system}.
\end{definition}

We now recall some basic results on hyperbolic systems from \cite{garding1959inequality, bauschke2001hyperbolic, renegar2006hyperbolic, gurvits2004combinatorics}. Consider a hyperbolic system $(\V, p, e)$ of degree $n$. Then, the map $\lambda: \V \to \Rn$ is Lipschitz (see \cite[Corollary 3.4]{bauschke2001hyperbolic}),
\begin{equation} \label{translation property of lambda}
	\lambda(e) = \one, \quad \text{and} \quad \lambda(x+te) = \lambda(x) + t\one\quad (\forall \,x \in \V,\,t \in \R).
\end{equation}

On $\V$, corresponding to the pair $(p,e)$, the \emph{hyperbolicity cone} and its interior are defined/denoted by
\[ \lplus(p,e) \coloneq \set{x \in \V}{\lambda(x) \geq 0} \quad \text{and} \quad \Lambda_{++}(p,e) \coloneq \set{x \in \V}{\lambda(x) > 0}. \]
When the context is clear, we denote these by $\lplus$ and $\Lambda_{++}$.

\gap
It is known, see \cite{garding1959inequality}, that
\begin{itemize}
	\item \textit{$\lplus(p, e)$ is a closed convex cone,}
	\item \textit{$\lplus(p,e) = \lplus(p,d)$ for all $d\in \Lambda_{++}(p,e)$,}
	\item \textit{$\lambda(x)=0\Leftrightarrow x\in \lplus(p,e) \cap (-\lplus(p,e))$, and}
	\item \textit{$\lplus(p, e)$ is pointed, that is, $\lplus(p,e) \cap (-\lplus(p,e)) = \{0\}$ if and only if $p$ is complete.}
\end{itemize}
Recall that a closed convex cone is said to be \emph{proper} or \emph{regular} if it is pointed and has nonempty interior. Since $e \in \Lambda_{++}(p, e)$, we see that
\begin{center}
	\begin{minipage}{0.85\textwidth}
		\emph{$(\V, p, e)$ is a complete hyperbolic system if and only if $\lplus(p, e)$ is a proper/regular cone.}
	\end{minipage}
\end{center}

\gap
Suppose $(\V,p,e)$  is a hyperbolic system. Then extending the definition of derivative polynomial in (\ref{derivative polynomial}),  we let 
\[ p^{(0)}=p,\,p^{(1)}=p^\prime,\, p^{(2)} = (p^\prime)^\prime,\,\ldots, \,p^{(n-1)} = (p^{(n-2)})'. \]
To avoid trivialities, henceforth, we assume that $n$ is at least $2$.

\begin{proposition} \textup{\cite{renegar2006hyperbolic}}
	Suppose $(\V,p,e)$  is a hyperbolic system of degree $n$. Then, $(\V,p^\prime,e)$ is a hyperbolic system whose hyperbolicity cone $\lplus^\prime$ contains $\lplus$. Moreover, when $n\geq 3$, $p^\prime$ is complete if and only if $p$ is complete.
\end{proposition}

Based on the validity of the Lax conjecture, Gurvits \cite[Proposition 1.2]{gurvits2004combinatorics} states the following:

\begin{proposition} \label{gurvits result1} 
	Consider a hyperbolic system $(\V, p, e)$  of degree $n$. Given $x, y \in \V$, there exist $n \times n$ real symmetric matrices $A$ and $B$ such that
	\[ \lambda(rx + sy) = \lambda(rA + sB) \quad (\forall\, r, s \in \R). \]
\end{proposition}

Note: For a real symmetric matrix $X$, we use the symbol $\lambda(X)$ to denote its vector of eigenvalues of $X$ written in decreasing order.

\gap
The result above is significant and useful. It allows us to make an eigenvalue statement for two elements in a hyperbolic system by knowing the corresponding statement for two real symmetric matrices. As noted in \cite{gurvits2004combinatorics}, we have the following.

\begin{proposition}  \label{gurvits result2}
	In a hyperbolic system of degree $n$, the following hold: For all $x, y \in \V$,
	\begin{itemize}
		\item[$(a)$] $\lambda(x) -\lambda(y) \prec \lambda(x-y)$,
		\item[$(b)$] $\lambda(x + y) \prec \lambda(x) + \lambda(y)$, 
		\item[$(c)$] $x \leq y$ implies $\lambda(x) \leq \lambda(y)$ in $\Rn$. (Here, $x \leq y$ means $y - x \in \lplus$.)
	\end{itemize}
\end{proposition}

Next, we recall the following construction of (semi-)inner product, due to Bauschke et al. \cite{bauschke2001hyperbolic}.

\begin{proposition} \label{bauschke2001hyperbolic} 
	\textup{\cite[Theorem 4.2 and Proposition 4.4]{bauschke2001hyperbolic}} 
	Consider a hyperbolic system $(\V, p, e)$ of degree $n$. Then, corresponding to $\lambda : \V \to \Rn$, 
	\begin{equation} \label{ip}
		\ip{x}{y} \coloneq \frac{1}{4} \Big[ \norm{\lambda(x+y)}^2 - \norm{\lambda(x-y)}^2\Big ] \quad (\forall\, x, y\in \V)
	\end{equation}
	defines a bilinear map with $\ip{x}{x} \geq 0$ for all $x$; hence defines a semi-inner product and induces the semi-norm $\norm{x} \coloneq \sqrt{\ip{x}{x}}$. Moreover,
	\begin{equation} \label{semi-ftvn condition}
		\ip{x}{y} \leq \ip{\lambda(x)}{\lambda(y)} \leq \norm{x} \norm{y} \quad (\forall\, x, y \in \V).
	\end{equation}
	When $p$ is complete, \eqref{ip} defines an inner product.
\end{proposition}

Note: While $\lambda$ and the semi-inner product \eqref{ip} depend on the pair $(p,e)$, for simplicity, we write them by suppressing $p$ and/or $e$. Also, we use the same inner product (norm) notation in $\V$ as well as in $\Rn$. \emph{Additionally, we speak of orthogonality and norm relative to the above semi-inner product.} For example, as a consequence of (\ref{semi-ftvn condition}), every element of $\lplus\cap \hspace{1pt} (-\lplus )$ is orthogonal to every element of $\V$.

\gap
We observe that \eqref{semi-ftvn condition} is equivalent to \emph{inner product-expanding and norm-preserving} conditions:
\begin{equation} \label{ip expanding} 
	\ip{x}{y} \leq \ip{\lambda(x)}{\lambda(y)} \quad \text{and} \quad \norm{x} = \norm{\lambda(x)} \quad (\forall\, x, y \in \V). 
\end{equation}
In a hyperbolic system $(\V,p,e)$, we define the \emph{trace} of an element $x$ in $\V$  by
\begin{equation} \label{trace}
	\tr(x) \coloneq \ip{x}{e} = \ip{\lambda(x)}{\one}=\text{sum of all the entries in $\lambda(x)$}. 
\end{equation}

Also, we let $(\lplus)^*$ denote the dual of $\lplus$ defined by 
\[ (\lplus)^* \coloneq \set{x \in \V}{\ip{x}{y} \geq 0, \,\, \forall\, y \in \lplus}. \]
We have the following result.

\begin{theorem}  \textup{(The hyperbolicity cone is a subdual cone)} \label{SG theorem}
	Consider a hyperbolic system $(\V, p, e)$ and the corresponding semi-inner product \eqref{ip}. Then, for all $x, y \in \lplus$, we have $\ip{x}{y} \geq 0$. Thus,
	\[ \lplus \subseteq (\lplus)^*. \]
\end{theorem}

\begin{proof}
	Fix $x, y \in \lplus$. From Proposition \ref{gurvits result1}, there exist $n \times n$ real symmetric matrices $A$ and $B$ such that $\lambda(x) = \lambda(A)$, $\lambda(y) = \lambda(B)$, $\lambda(x + y) = \lambda(A + B)$ and $\lambda(x - y) = \lambda(A - B)$. As $\lambda(A) = \lambda(x) \geq 0$ and $\lambda(B) = \lambda(y) \geq 0$, we see that $A$ and $B$ are positive semidefinite matrices. It follows (from the note below) that $\trace(AB) \geq 0$. Then, 
	\[ \ip{x}{y} = \frac{1}{4} \Big[ \norm{\lambda(A + B)}^2 - \norm{\lambda(A - B)}^2 \Big] = \trace(AB) \geq 0. \]
	This completes the proof.
\end{proof}

Note: In the Euclidean Jordan algebra $\Sn$ of all $n \times n$ real symmetric matrices, the inner product is given by $\ip{X}{Y} \coloneq \trace(XY)$ ($=$ the sum of diagonal entries of $XY$ $=$ the sum of all eigenvalues of $XY$). So, $\norm{\lambda(X)}^2 = \norm{X}^2$. Moreover, when $X$ and $Y$ are positive semidefinite, writing $Z \coloneq \sqrt{X}\sqrt{Y},$ we have 
\begin{align*}
	\trace(XY) &= \trace(\sqrt{X}\sqrt{X}\sqrt{Y}\sqrt{Y}) \\
	& = \trace(\sqrt{X}\sqrt{Y}\sqrt{Y}\sqrt{X}) = \trace(ZZ^T) \geq 0.
\end{align*} 
Consequently, when both $X$ and $Y$ are positive semidefinite and $\trace(XY)=0$, we have $XY = 0 = YX$; in particular, $X$ and $Y$ commute.

\begin{definition} \label{rank def}
	Let $(\V, p, e)$ be a hyperbolic system. For any $x \in \V$, we define $\rank(x) \coloneq$ the number of nonzero entries in $\lambda(x)=$ the degree of $p(e+tx)$ (as a polynomial in $t$). We say that an element $c\in \V$ is of rank one or a rank-one element if $\rank(c)=1$.
\end{definition}

\begin{remark}
	An observation due to Gurvits \cite{gurvits2005combinatorial} (see also Renegar \cite{renegar2006hyperbolic}) says that in $(\V,p,e)$, $\rank(x)$ is independent of the direction $e$, i.e., one could replace $e$ by any other $d\in \Lambda_{++}$. However, when $\V$ and $e$ are fixed, there could be several hyperbolic polynomials inducing (the same) $\lplus$. In this setting, $\rank(x)$ depends on the choice of the hyperbolic polynomial; see Examples \ref{example3.3} and \ref{example3.4} below.
\end{remark}

\begin{proposition} \label{ito2023automorphisms} 
	In a hyperbolic system, the following hold:
	\begin{itemize}
		\item $\max\{\rank(x), \rank(y)\}\leq \rank(x+y)$ for all $x, y \in \lplus$.
		\item $\rank(x + y) \leq \rank(x) + \rank(y)$ for all $x, y \in \V$.
	\end{itemize}
\end{proposition}

\begin{proof} 
	The first inequality is a consequence of Item $(c)$ in Proposition \ref{gurvits result2}; the second inequality can be proved via Proposition \ref{gurvits result1} by noting that such a rank inequality holds for real symmetric matrices. (We note that the second rank inequality is also proved in \cite{ito2023automorphisms}, Proposition 2.4 for $x,y\in \lplus$.)
\end{proof}

One consequence of the above result is that 
\[ \rank(y) = 0 \implies \rank(x+y) = \rank(x) + \rank(y) = \rank(x) \quad (\forall\, x \in \lplus). \]
In our next result, we improve this by showing that rank is additive on mutually orthogonal elements of $\lplus$. It is essentially a consequence of the result that for positive semidefinite matrices $A$ and $B$ with $\trace(AB)=0$, the eigenvalue vector of $A+B$ consists of the nonzero eigenvalues of $A$ and $B$, and possibly some zeros.

\begin{theorem} \label{rank additive result}
	\textup{(Rank additivity on orthogonal elements of $\lplus$)}
	Consider a hyperbolic system $(\V, p, e)$ of degree $n$ and the corresponding semi-inner product \eqref{ip}.  Let $a, b \in \lplus$ with $\ip{a}{b} = 0$. Then the following statements hold:
	\begin{itemize}
		\item[$(i)$] If $\lambda(a)=0$, then $\lambda(a+b)=\lambda(b)$.
		\item[$(ii)$] If $\lambda(a)\neq 0$ with $\lambda(a) = (r_1, r_2, \ldots, r_k, 0, 0, \ldots, 0)^T$, where $r_i > 0$ for $1 \leq i \leq k \leq n$, then $\lambda(a+b)$ is of the form
		\[ \lambda(a+b) = \big[ (r_1, r_2, \ldots, r_k, s_{k+1}, s_{k+2}, \ldots, s_n)^T \big]^\downarrow, \]
		where each $s_i$ is an eigenvalue of $b$ and every nonzero eigenvalue of $b$ appears in the set $\{s_{k+1}, s_{k+2}, \ldots, s_n\}$.
		\item[$(iii)$] $ \rank(a+b) = \rank(a) + \rank(b). $
	\end{itemize}
	Moreover, if $a_1, a_2, \ldots, a_N$ are mutually orthogonal elements of $\lplus$, then
	\begin{equation} \label{rank sum}
		\rank(a_1 + a_2 + \cdots + a_N) = \rank(a_1) + \rank(a_2) + \cdots + \rank(a_N).
	\end{equation}
\end{theorem}

\begin{proof}
	By Proposition \ref{gurvits result1}, there exist $n \times n$ real symmetric matrices $A$ and $B$ such that 
	\begin{equation} \label{gurvits}
		\lambda(ta+sb) = \lambda(tA+sB),
	\end{equation}
	for all $t, s \in \R$. In particular, $\lambda(A) = \lambda(a)$ and  $\lambda(B) = \lambda(b)$. Now, as $a, b \in \lplus$, $\lambda(a)$ and $\lambda(b)$ are nonnegative. It follows that $A$ and $B$ are positive semidefinite matrices. From \eqref{ip}, 
	\[ 0 = \ip{a}{b} = \ip{A}{B} = \trace(AB). \]
	This implies that $A$ and $B$ commute in the usual sense (see the statement made before Definition \ref{rank def}); consequently, there exists an orthogonal matrix $U$ such that
	\[ A = UDU^T \quad \text{and} \quad B = UEU^T, \]
	where $D$ and $E$ are diagonal matrices with diagonal entries of $D$ and $E$ coming from entries of $\lambda(A)$ and $\lambda(B)$ respectively. We now justify items $(i)- (iii)$.
	
	$(i)$ Suppose $\lambda(a)=0$. Then, $\lambda(A)=0$. As $A$ is positive semidefinite, $A=0$. Then, from \eqref{gurvits}, $\lambda(a+b) = \lambda(A+B) = \lambda(B) = \lambda(b)$.
	
	$(ii)$ Here $0\neq \lambda(a)=(r_1, r_2, \ldots, r_k, 0, 0, \ldots, 0)^T=\lambda(A)$. Consider the diagonal matrices $D$ and $E$ as above. By permuting the diagonal entries of $D$ if necessary, i.e., changing $U$, we may assume that the diagonal vector of $D$ is written as $\operatorname{diag}(D) = (r_1, r_2, \ldots, r_k, 0, 0, \ldots, 0)^T$. When we do this, we may write 
	$\operatorname{diag}(E) = s = (s_1, s_2, \ldots, s_n)^T$ so that $\lambda(b) = \lambda(B) = s^\downarrow$, i.e., each $s_i$ is an eigenvalue of $b$. Then,  
	\[ 0 = \ip{A}{B} = \ip{D}{E} = \trace(DE) = r_1s_1 + r_2s_2 + \cdots + r_ks_k. \]
	Since all $r_i$ are positive and $s_i$ are nonnegative, we see that $s_i = 0$ for $1 \leq i \leq k$. This implies that the nonzero eigenvalues of $b$ appear in the set
	$\{s_{k+1}, s_{k+2}, \ldots, s_n\}$. Now, we have $A+B = U(D+E)U^T$, where $\operatorname{diag}(D + E) = (r_1, r_2, \ldots, r_k, s_{k+1}, \ldots, s_n)^T$. It follows that 
	\[ \lambda(a+b) = \lambda(A+B) = \big[ (r_1, r_2, \ldots, r_k, s_{k+1}, \ldots, s_n)^T \big]^\downarrow. \]
	This proves $(ii)$.
	
	$(iii)$ In view of  Items $(i)$ and $(ii)$, by counting the number of nonzero elements in $\lambda(a + b)$, we see that $\rank(a + b) = \rank(a) + \rank(b)$.\\
	Finally, \eqref{rank sum} comes from an induction argument. (For the induction step, say, to go from $m$ to $m+1$, let $a \coloneq a_1 + a_2 + \cdots + a_m$ and $b \coloneq a_{m+1}$.)
\end{proof}

\begin{definition}\label{automorphism definition}
	Suppose $(\V, p, e)$ is a hyperbolic system. An invertible linear transformation $A : \V \to \V$ is said to be a
	\begin{itemize}
		\item \emph{system-automorphism} if $\lambda(Ax) = \lambda(x)$ for all $x \in \V$. We write $\Aut(\V,p,e)$ $($in short, $\Aut(\V))$ for the group of all such automorphisms;
		\item \emph{cone-automorphism} if $A(\lplus) = \lplus$. We write $\Aut(\lplus)$ for the group of all such automorphisms.
	\end{itemize}
\end{definition}

\begin{proposition}\label{prop properties of automorphism}
	In a hyperbolic system, we have the following:
	\begin{itemize}
		\item[$(i)$] Every system-automorphism is inner-product-preserving.
		\item[$(ii)$] $\Aut(\V) \subseteq \Aut(\lplus)$.
		\item[$(iii)$] Every system-automorphism preserves the rank of any $x \in \V$.
		\item[$(iv)$] When $p$ is a minimal polynomial, every cone-automorphism preserves the rank of any $x \in \lplus$. 
	\end{itemize}
\end{proposition}

\begin{proof}
	Item $(i)$ comes from (\ref{ip}). Items $(ii)$ and $(iii)$ are obvious. Finally, Item $(iv)$ comes from \cite[Proposition 2.8]{ito2023automorphisms}.
\end{proof}

We now define the concept of a \emph{doubly stochastic transformation} on a hyperbolic system.

\begin{definition} \label{definition: ds}
	Given a hyperbolic system $(\V, p, e)$, a linear transformation $T : \V \to \V$ is said to be \emph{doubly stochastic} if 
	\[ T(\lplus) \subseteq \lplus, \quad  T(e) = e, \quad  \text{and} \quad \tr(T(x)) = \tr(x) \quad (\forall\, x \in \V). \]
\end{definition}

The above properties are respectively called the positivity, unital, and trace-preserving properties. We note that when $p$ is complete (in which case \eqref{ip} defines an inner product), the trace condition above is equivalent to $T^*e = e$, where $T^*$ denotes the adjoint of $T$. Additionally, when the system comes from a Euclidean Jordan algebra, $T$ is doubly stochastic if and only if $T^*$ is doubly stochastic. This is because a symmetric cone (the hyperbolicity cone in a Euclidean Jordan algebra) is self-dual.

\section{Rank-one elements, idempotents, and minimal polynomials}

Recall that in a hyperbolic system $(\V, p, e)$, an element $c$ has rank one if $c$ has just one nonzero eigenvalue. In particular,  $c\in \lplus$ is a rank-one element if 
$\lambda(c) = (\alpha, 0, 0, \ldots, 0)^T$, where $\alpha>0$. 

\begin{definition}
	Let $(\V, p, e)$ be a hyperbolic system of degree $n$. Let $\lambda: \V \to \Rn$ be the corresponding eigenvalue map. We say that an element $c\in \V$ is 
	\begin{itemize} 
		\item an \emph{idempotent} if $\lambda(c) = (1, 1, \ldots, 1, 0, 0, \ldots, 0)^T = \one_k$ for some $k$, $1 \le k \le n$. 
		\item a \emph{primitive idempotent} if $\lambda(c) = (1, 0, 0, \ldots, 0)^T$.
	\end{itemize}
\end{definition}

\emph{Note that the above concepts are restricted to nonzero elements of $\lplus$} and depend on $p$ and $e$. A rank-one element of $\lplus$ is a positive scalar multiple of a primitive idempotent. As we see below (in Example \ref{example3.2}), the existence of a rank-one element is not guaranteed. However, $e$ is (always) an idempotent as $\lambda(e) = \one$. Since $\norm{x} = \norm{\lambda(x)}$ for all $x\in \V$, the norm of any primitive idempotent is $1$ and the norm of $e$ is $\sqrt{n}$.

\begin{example} \label{example3.2} 
	On $\R$, let $e=1$ and $p(x) = x^2$. Then, $\lambda(x) = (x,x)^T \in \R^2$. In the complete hyperbolic system $(\R, p, 1)$, there are no rank-one elements. In particular, the idempotent $e$ is not a sum of primitive idempotents.
\end{example}

\begin{example} \label{example3.3} 
	On $\R^3$, for $x = (x_1,x_2,x_3)^T$, let $p(x) \coloneq x_1x_2x_3$ and $e = (1, 1, 1)^T$. Then, $(\R^3, p, e)$ is a complete hyperbolic system with $\lambda(x)=x^\downarrow \in \R^3$. Correspondingly, $\ip{x}{y} = x_1y_1 + x_2y_2 + x_3y_3$. In this system, the standard coordinate vectors $e_1$, $e_2$, and $e_3$ are primitive idempotents.
\end{example}

\begin{example} \label{example3.4} 
	On $\R^3$, for $x = (x_1, x_2, x_3)^T$, let $p(x) \coloneq x_1^2x_2x_3$ and $e = (1, 1, 1)^T$. Then, $(\R^3, p, e)$ is a complete hyperbolic system with $\lambda(x)=[(x_1, x_1, x_2, x_3)^T]^\downarrow \in \R^4$. Correspondingly, $\ip{x}{y} = 2x_1y_1 + x_2y_2 + x_3y_3$. In this system, the standard coordinate vectors $e_2$ and $e_3$ are primitive idempotents, but $e_1$ is not (as it has rank $2$).
\end{example}

\begin{example} \label{example3.5} 
	On $\R^3$, for $x=(x_1, x_2, x_3)^T$, let $p(x) \coloneq x_1x_2$ and $e = (1, 1, 1)^T$. Then, $(\R^3, p, e)$ is a hyperbolic system with $\lambda(x) = [(x_1, x_2)^T]^\downarrow \in \R^2$.
	Correspondingly, $\ip{x}{y} = x_1y_1 + x_2y_2$. In this system, the standard coordinate vectors $e_1$ and $e_2$ are primitive idempotents.
\end{example}

The following example, taken from \cite{nagano2024projection}, shows that even in a minimal hyperbolic system, there may not be any rank-one element (in particular, primitive idempotent).

\begin{example} \label{example3.6}
	Consider the polynomial $p : \R^3 \to \R$ defined by
	\[ p(x) = (x_1 + x_2 + x_3)(x_1 - x_2 + x_3)(2x_1 - x_2 - x_3)(x_1 + 2x_2 - x_3). \]
	Then $p$ is a minimal hyperbolic polynomial of degree $4$ in the direction $e = (0, 0, 1)^T$; see \cite[Proposition 3.10]{nagano2024projection}. For any $x \in \R^3$, we note that the roots of the polynomial $t \mapsto p(te - x)$ are
	\begin{align*}
		r_1 &= x_1 + x_2 + x_3, & r_2 &= x_1 - x_2 + x_3, \\
		r_3 &= -2x_1 + x_2 + x_3, & r_4 &= -x_1 - 2x_2 + x_3.
	\end{align*}
	Now assume that any three of the roots $\{r_1, r_2, r_3, r_4\}$ are zero. Then the associated homogeneous linear system has only the trivial solution $x_1 = x_2 = x_3 = 0$, which forces the remaining root to be zero as well. Thus, there is no rank-one element (in particular, no primitive idempotent).
\end{example}

\begin{example} 
	Consider a Euclidean Jordan algebra $\V$ of rank $n$ with unit $e$. As described in the Introduction, let $\lambda: \V \to \Rn$ denote the eigenvalue map. With $p(x) = \det (x)$, which is the product of the entries of $\lambda(x)$, $(\V, p, e)$ becomes a complete and minimal hyperbolic system of degree $n$ (the minimality of $p$ is shown in \cite{ito2023automorphisms}; also see Corollary  \ref{ito2023automorphisms corollary} below); the induced inner product is the trace inner product $\ip{x}{y} = \tr(x \circ y)$. The primitive idempotents in the algebra $\V$ are exactly the primitive idempotents in the corresponding hyperbolic system. Thanks to the spectral decomposition theorem, every element is a linear combination of primitive idempotents and, in particular, every idempotent is a sum of primitive idempotents. In the Euclidean Jordan algebra $\Rn$ (which carries the usual inner product and componentwise product with $\one$ as the unit), the standard coordinate vectors are the only primitive idempotents.
\end{example}

Here is an elementary result regarding rank-one elements.

\begin{proposition} \label{primitive idempotents} 
	In a complete hyperbolic system, every rank-one element in $\lplus$ induces an extreme, exposed direction of $\lplus$.     
\end{proposition}

\begin{proof}
	Assume that our system is complete and $c$ is a rank-one element in $\lplus$. Then,  \eqref{ip} defines an inner product.
	By scaling, we may assume that $c$ is a primitive idempotent. Suppose $c = x + y$, where $x, y \in \lplus$. We show that $x$ and $y$ are multiples of $c$. First, we observe that $\lambda(x)$ and $\lambda(y)$ are nonnegative vectors. Now, from Proposition \ref{gurvits result2}, $\lambda(c) = (1, 0, 0, \ldots, 0)^T = \lambda(x + y) \prec \lambda(x) + \lambda(y)$. Then, by the definition of majorization in $\Rn$, $\lambda(x) + \lambda(y) = (1, 0, 0, \ldots, 0)^T = \lambda(c)$. Furthermore,  nonnegativity of $\lambda(x)$ and $\lambda(y)$ imply that $\lambda(x) = t\lambda(c)$ and $\lambda(y) = (1-t)\lambda(c)$ for some $t \in [0,1]$. It follows that $\norm{x} = t$ and $\norm{y} = 1-t$. Now,
	\[ 1 = \norm{c} = \norm{x + y} \leq \norm{x} + \norm{y} = t + (1-t) = 1. \]
	By the strict convexity of the (inner product) norm, $x$ and $y$ are nonnegative multiples of $c$. This shows that $c$ is an extreme direction of $\lplus$. That $c$ is exposed comes from a result of Renegar \cite[Theorem 23]{renegar2006hyperbolic}: \emph{All boundary faces of $\lplus$ are exposed.}
\end{proof}

Now consider a hyperbolic system $(\V,p,e)$ with hyperbolicity cone $\lplus$; let $q$ be the corresponding minimal polynomial  (so $q$ is hyperbolic in direction $e$, induces $\lplus$, and has the least degree among such polynomials). As mentioned in the Introduction, thanks to the result of Helton and Vinnikov \cite{helton2007lmi}, $q$ is unique up to a positive constant and divides $p$. Let us call $(\V,q,e)$ \emph{the minimal system} of $(\V, p, e)$ and $q$ a \emph{minimal polynomial}. Given a hyperbolic system $(\V,p,e)$, the question of when $p$ is minimal becomes interesting. In Theorem \ref{condition for minimality}, we provide an answer - thereby improving a result on ROG-cones \cite{ito2023automorphisms}. First, we cover some preliminary results.

\begin{proposition} \label{minimal system}
	Suppose $q$ is a minimal polynomial of the hyperbolic system $(\V,p,e)$. Then, the following hold:
	\begin{itemize}
		\item[$(i)$] $p$ is complete if and only if $q$ is complete. 
		\item[$(ii)$] For any $x \in \V$, $\lambda_q(x)$ is a `subvector' of $\lambda_p(x)$; furthermore, $\rank_q(x) \leq \rank_p(x)$.
	\end{itemize}
\end{proposition}

\begin{proof}
	$(i)$ This follows from the fact that completeness is equivalent to the pointedness of the hyperbolicity cone; see Section 2.\\
	$(ii)$ Thanks to Helton and Vinnikov \cite{helton2007lmi}, we can write $p=qh$ for some  polynomial $h$. As $p(te-x)=q(te-x)\,h(te-x)$, we see that the eigenvalues of $x$ relative to $q$ (including multiplicity) form a subset of the eigenvalues of $x$ relative to $p$. Hence, $(ii)$ holds.
\end{proof}

\begin{proposition} \label{hereditary property of rank-one element}
	Consider a hyperbolic system $(\V,p,e)$ of degree $n$. Corresponding to $p$, let $q$ and $p^\prime$ denote, respectively, its minimal polynomial and derivative polynomials. Let $c\in \lplus$ be a rank-one element in $(\V,p,e)$. Then the following statements hold:
	\begin{itemize}
		\item[$(a)$] For any $d \in \Lambda_{++}$, $c$ is a rank-one element in $(\V,p,d)$.
		\item[$(b)$]  $c$ is a rank-one element in $(\V,q,e)$; in fact, except for possibly having different sizes/dimensions,  $\lambda_q(c)=(\alpha,0,0,\ldots, 0)^T$ when $\lambda_p(c)=(\alpha,0,0,\ldots, 0)^T$.
		\item[$(c)$] When $n\geq 2$, $c$ is a rank-one element in $(\V,p^\prime,e)$. 
	\end{itemize}
\end{proposition}

\begin{proof} 
	$(a)$ This comes from  \cite[Proposition 22]{renegar2006hyperbolic}, where it is shown that the rank of an element $x\in \lplus$ is independent of direction $d\in \Lambda_{++}$.
	
	$(b)$ As $c$ is of rank-one in $(\V,p,e)$, $\lambda_p(c)=(\alpha, 0, \ldots, 0)^T$, where $\alpha>0$. Recall from \cite[Theorem 3]{garding1959inequality} that for an $x\in \V$, $\lambda_p(x) = 0$ if and only if $x \in \Lambda_{+}(p, e) \cap (-\Lambda_{+}(p, e))$. Since $\Lambda_{+}(p, e) = \Lambda_{+}(q, e)$, we have $\Lambda_{+}(p, e) \cap (-\Lambda_{+}(p, e)) = \Lambda_{+}(q, e) \cap (-\Lambda_{+}(q, e))$; thus $\lambda_q(c) \neq 0$. Since $\lambda_q(c)$ is a `subvector' of $\lambda_p(c)$, we see that $\alpha$ appears exactly once (with all other entries zero) in $\lambda_q(c)$. This shows that $c$ is of rank one in $(\V, q, e)$.
	
	$(c)$ For $n\geq 3$, this comes from \cite[Theorem 12]{renegar2006hyperbolic}, where it is proved in terms of the multiplicity function, $\text{mult}(x) \coloneq n-\rank(x)$. We repeat the proof here (which works even for $n=2$), highlighting its elementary and useful argument.  Let $\lambda^\prime$ be the eigenvalue map corresponding to $(\V,p^\prime,e)$. By the interlacing property between the eigenvalues relative to $\lambda$ and $\lambda^\prime$ (see \cite[page 67]{renegar2006hyperbolic}) we have:
	\[ \lambda_1(c) \geq \lambda_1^{\prime}(c) \geq \lambda_2(c) \geq \cdots \geq \lambda_{n-1}^{\prime}(c) \geq \lambda_n(c), \]
	with
	\begin{equation}\label{equality in interlacing}
		\big[ \lambda_j(c) = \lambda_{j}^\prime(c) \,\, \text{or} \,\, \lambda_{j}^\prime(c) = \lambda_{j+1}(c) \big] \iff \lambda_j(c) = \lambda_{j}^\prime(c) = \lambda_{j+1}(c).
	\end{equation}
	When $\lambda(c)=(\alpha,0,0\ldots,0)^T$, we have
	\[ \alpha \geq \lambda_1^{\prime}(c) \geq 0 \geq \lambda_2^\prime(c) \geq 0 \cdots \geq \lambda_{n-1}^{\prime}(c) \geq 0. \]
	Then $\lambda_2^\prime(c) = \lambda_3^\prime(c) = \cdots = \lambda_{n-1}^\prime(c) = 0$, and because of \eqref{equality in interlacing}, $\alpha > \lambda_1^\prime(c) > 0$.  Letting $\beta = \lambda_1^\prime(c)>0$, 
	we have $\lambda^\prime(c) = (\beta, 0, 0, \ldots, 0)^T$. So $c$ is a rank-one element in $(\V, p^\prime, e)$.
\end{proof}

While Item $(c)$ in the above result is for a rank-one element, we state a broader result as follows. This is essentially a restatement of \cite[Theorem 12]{renegar2006hyperbolic} in terms of rank.

\begin{proposition} \label{rank preservation}
	Let $(\V, p, e)$ be a hyperbolic system of degree $n$. 
	\begin{itemize}
		\item[$(a)$] Suppose $c \in \lplus$ has rank $k$ in $(\V,p,e)$, where $1 \leq k \leq n-1$. Then $c \in \lplus^{\prime}$ and has rank $k$ in $(\V,p^\prime,e)$.
		\item[$(b)$] Suppose $c \in \lplus^{\prime}$ has rank $k$ in $(\V,p^\prime,e)$, where $1 \leq k \leq n-3$. Then $c \in \lplus$ and has rank $k$ in $(\V,p,e)$.
	\end{itemize}
\end{proposition}

\begin{proof}
	$(a)$ Since $\lplus \subseteq \lplus^{\prime}$ we have $c\in \lplus^{\prime}$. Now, as $c$ has rank $k$ in $(\V,p,e)$, $\lambda_{k}(c) > 0$ and $0=\lambda_{k+1}(c) = \lambda_{k+2}=\cdots=\lambda_{n}(c)$. Then, we must have $\lambda'_{k}(c) > 0$, otherwise the interlacing property along with \eqref{equality in interlacing} would imply $\lambda_{k}(c) = 0$, yielding a contradiction. Since $\lambda_{k+1}(c)=0\geq \lambda_{k+1}^\prime(c)\geq \cdots \geq \lambda_{n}(c)\geq 0$, we conclude that $c$ has rank $k$ in $(\V, p^\prime, e)$.
	
	$(b)$ As $1\leq k\leq n-3$, we have $\lambda'_{n-2}(c) = \lambda'_{n-1}(c) = 0$. Thus we have $\lambda_{n-1}(c) = 0$ by the interlacing property. Since we now have $\lambda_{n-1}(c) = \lambda'_{n-1}(c) = 0$, from  \eqref{equality in interlacing} we observe that $\lambda_n(c) = 0$. This shows that $c \in \lplus$. We now show that $c$ has rank $k$ in $(\V,p,e)$. Note that $\lambda'_{k}(c) > 0$ and $\lambda'_{k+1}(c) = \lambda'_{k+2}(c) = 0$. Since $\lambda_k(c) \geq \lambda'_k(c) > 0$, we only need to show that $\lambda_{k+1}(c) = 0$. Indeed, $\lambda_{k+2}(c) = 0$ by the interlacing property and, since $\lambda'_{k+1}(c) = \lambda_{k+2}(c) = 0$, \eqref{equality in interlacing} forces $\lambda_{k+1}(c) = 0$ as well.
\end{proof}

We end this section by stating an important result on the minimality of $p$.

\begin{theorem} \label{condition for minimality}
	\textup{(Minimality of $p$ and $p^\prime$)}
	Suppose $(\V,p,e)$ is a hyperbolic system of degree $n$ in which there exists a finite set $\{c_1,c_2,\ldots, c_k\}$ of elements in $\lplus$ such that
	\begin{itemize}
		\item each $c_i$ is of rank one, and
		\item $c_1 + c_2 + \cdots + c_k \in \Lambda_{++}$.
	\end{itemize}
	Then $p$ is minimal. Moreover, when $n\geq 2$, $p^\prime$ is minimal and $\lplus$ is strictly contained in $\lplus^\prime$.
\end{theorem}

\begin{proof} 
	First, we show that $p$ (with degree $n$) is minimal. Let $q$ be the minimal polynomial corresponding to $(\V,p,e)$; let $m$ be its degree.  We will show that $m=n$. Let $d \coloneq c_1 + c_2 + \cdots + c_k\in \Lambda_{++}$ and $\mu$ be the eigenvalue map corresponding to $(\V, p, d)$. By our previous result, each $c_i$ has rank one in the system $(\V,p,d)$; let $\mu(c_i) = (\alpha_i, 0, \ldots, 0)^T\in \Rn$, where $\alpha_i > 0$. 
	Now, in the hyperbolic system $(\V,p,d)$ and the induced (semi-)inner product, $\tr(c_i) = \ip{d}{c_i} = \alpha_i$ for all $i$; hence, by the linearity of the trace, 
	\begin{equation}\label{trace equality} n = \tr(d) = \sum_{i=1}^{k} \tr(c_i) = \sum_{i=1}^{k}\alpha_i. 
	\end{equation}
	(This also follows from  Proposition \ref{gurvits result2}: $\one_n = \mu(d) \prec \sum_{i=1}^{k}\mu(c_i)$.)
	
	Consider the system $(\V, q, d)$ with its eigenvalue map $\nu$. By the previous result, $\nu(c_i)=(\alpha_i, 0, \ldots, 0)^T \in \R^m$. Noting that the argument used in the derivation of (\ref{trace equality}) works in $(\V,q,d)$ as well, we get $m = \sum_{i=1}^{k} \alpha_i$.  Thus $m=n$, proving the minimality of $p$. 
	
	Now we prove the minimality of $p^\prime$ when $n\geq 2$. By the previous result, each $c_i$ continues to be of rank one in $(\V,p^\prime,e)$, and the sum of all these, namely, $d$, is in $\Lambda_{++}$. As $\lplus \subseteq \lplus^\prime$, see \cite{renegar2006hyperbolic}, we have $\Lambda_{++}\subseteq \Lambda^\prime_{++}$. By the first part, $p^\prime$ is minimal.
	
	Finally, if $\lplus = \lplus^\prime$, then $p$ and $p^\prime$ generate the same hyperbolicity cone. This contradicts the minimality of $p$ as $p^\prime$ has degree less than that of $p$. Hence, $\lplus$ is strictly contained in $\lplus^\prime$. 
\end{proof}

\section{(Scaled) Jordan frames}

Motivated by Theorem \ref{condition for minimality}, in this section, we introduce the concept of a (scaled) Jordan frame. The usage of the phrase `Jordan frame' as opposed to just a `frame' is intentional for two reasons. First, the concept of a `frame' exists in the mathematical literature (in a finite-dimensional setting, it is a set that spans the entire space). Second, as shown below, our `Jordan frame' in a hyperbolic setting reduces to that of a Jordan frame in a Euclidean Jordan algebra.

\begin{definition}
	In a hyperbolic system $(\V,p,e)$, a finite set $\{c_1, c_2, \ldots, c_k\}$ of elements in $\lplus$ is said to be a 
	\begin{itemize}
		\item \emph{Jordan frame} if each $c_i$ is a primitive idempotent and 
		\[ c_1 + c_2 + \cdots + c_k = e; \]
		\item \emph{scaled Jordan frame} if each $c_i$ is a rank-one element and 
		\[ c_1 + c_2 + \cdots + c_k = d\in \Lambda_{++}. \]
	\end{itemize}
\end{definition}

\begin{remark}
	Note that in a scaled Jordan frame, objects may repeat or be proportional to each other. However, as we see below, $k\geq n$. We also see that a Jordan frame is necessarily orthonormal (hence linearly independent) with $n$ elements.
\end{remark}

\begin{remark}
	The existence of a (scaled) Jordan frame is not always guaranteed; see Example \ref{example3.2}. However, as our proof of Corollary \ref{ito2023automorphisms corollary} shows, in a complete hyperbolic system where the hyperbolicity cone is ROG, there are always scaled Jordan frames.
	\emph{We also observe that a system-automorphism takes a primitive idempotent to a primitive idempotent} and (when the system is complete) takes a Jordan frame to a Jordan frame. Moreover, \emph{when $p$ is minimal, every cone-automorphism maps a scaled Jordan frame to a scaled Jordan frame}; see Proposition \ref{prop properties of automorphism}.
\end{remark}

\subsection{Properties of a scaled Jordan frame}

\begin{theorem}\label{thm on scaled JF}
	Suppose in a hyperbolic system $(\V, p, e)$ of degree $n$,  $\{c_1, c_2, \ldots, c_k\}$ is a scaled Jordan frame. Then $k \geq n$. Moreover, $k = n$ if and only if  $\{c_1, c_2, \ldots, c_k\}$ is a Jordan frame in $(\V, p, d),$ where $d \coloneq c_1+c_2+\cdots+c_k$.
\end{theorem}

\begin{proof}
	We first show that $k \geq n$.  As $\{c_1, c_2, \ldots, c_k\}$ is a scaled Jordan frame in $(\V, p, e)$, $d = c_1 + c_2 + \cdots + c_k \in \Lambda_{++}$. By Proposition \ref{hereditary property of rank-one element}, each $c_i$ has rank one in $(\V, p, d)$. Let $\mu$ be the eigenvalue map corresponding to $(\V, p, d)$ so $\mu(c_i) = (\alpha_i, 0, \ldots, 0)^T$, where $\alpha_i > 0$.  Working with the (semi-)inner product induced by $(\V,p,d)$, by (\ref{trace equality}), 
	\[ n = \tr(d) = \sum_{i=1}^{k} \tr(c_i) = \sum_{i=1}^{k}\alpha_i. \]
	Next, for all $i$,
	\begin{equation} \label{alphai equality}
		\alpha_i = \tr(c_i) = \ip{d}{c_i} = \norm{c_i}^2 + \sum_{j\neq i} \ip{c_i}{c_j} = \alpha_i^2 + \sum_{j \neq i} \ip{c_i}{c_j}. 
	\end{equation}
	As $\ip{c_i}{c_j} \geq 0$ for all $i, j$ (from Theorem \ref{SG theorem}), we see that $\alpha_i \geq \alpha_i^2$, i.e., $\alpha_i \leq 1$. Then
	\[ k =\sum_{i=1}^{k} 1\geq \sum_{i=1}^{k} \alpha_i = n. \]
	Now suppose $k=n$. From the above inequality, $\alpha_i = 1$ for all $i$. Thus, each $c_i$ is a primitive idempotent in $(\V, p, d)$. Since the sum of the $c_i$'s is $d$, we see  that $\{c_1,c_2,\ldots, c_k\}$ is a Jordan frame in $(\V, p, d)$.
	
	To see the reverse implication, suppose each $c_i$ is primitive in $(\V, p, d)$ so $\alpha_i=1$ for all $i$. Then, $n=\sum_{i=1}^{k}\alpha_i=k$. 
\end{proof}

\begin{definition} 
	In the setting of hyperbolic systems, we say that a property is \emph{hereditary} if the property carries over from a (hyperbolic) polynomial to its derivative polynomial.
\end{definition}

For example, when $n \geq 3$, the completeness of $p$ is hereditary. This is because, for $n \geq 3$, the regularity of $\lplus$ implies that of $\lplus^\prime$, see \cite[Proposition 13]{renegar2006hyperbolic}. We show below that the property of having a scaled Jordan frame is hereditary, whereas the property of having an ROG-cone is not hereditary.

\gap
Our next result deals with the $m$th derivative polynomial of $p$; see Section 2 for its definition.

\begin{theorem} \label{scaled JF property for the mth derivative}
	Suppose $(\V, p, e)$ is a hyperbolic system of degree $n \geq 2$ and $0 \leq m\leq n-1$. Then every scaled Jordan frame in $(\V, p, e)$ is a scaled Jordan frame in $(\V, p^{(m)}, e)$. Consequently, if $(\V,p,e)$ has a scaled Jordan frame, then $p^{(m)}$ is minimal. 
\end{theorem}

\begin{proof}
	Suppose $\{c_1,c_2,\ldots, c_k\}$ is a scaled Jordan frame in $(\V,p,e)$. Proposition \ref{hereditary property of rank-one element}(c) shows that each $c_i$ continues to be of rank one in $(\V,p^\prime,e)$; moreover, the sum of all these is in $\Lambda_{++}\subseteq \Lambda_{++}^\prime$. Thus, $\{c_1,c_2,\ldots, c_k\}$ is a scaled Jordan frame in $(\V,p^\prime,e)$. Now an induction argument shows that $\{c_1,c_2,\ldots, c_k\}$ is a scaled Jordan frame in any subsequent derivative system $(\V, p^{(m)}, e)$. Finally, the minimality of $p^{(m)}$ follows from Theorem \ref{condition for minimality}. 
\end{proof}

\begin{corollary} \label{ito2023automorphisms corollary} 
	\textup{\cite[Propositions 3.5 and 3.8]{ito2023automorphisms}}
	Suppose $(\V, p, e)$ is a hyperbolic system, where  $\lplus$ is an $\mathrm{ROG}$-cone. Then all the derivative systems $(\V, p^{(m)},e)$ are minimal. The same conclusion holds if $(\V,p,e)$ comes from a Euclidean Jordan algebra.
\end{corollary}

\begin{proof} 
	We offer a proof based on Theorem \ref{scaled JF property for the mth derivative}.
	
	Suppose first that $\lplus$ is an ROG-cone. We show that $(\V,p,e)$ carries a scaled Jordan frame. For this, we use the following well-known result: In an $n$-dimensional inner product space, for a compact convex set $K$, every element $x\in K$ is a convex combination of at most $n+1$ extreme points; see \cite[page 84]{rudin1991functional}. Now, as our cone $\lplus$ is proper, we can take $K = \set{x \in \lplus}{\ip{x}{e} = 1}$ to be a base of $\lplus$ and see that every element in $\lplus$ is a convex combination of extreme directions (these come from the extreme points of $K$). Since $\lplus$ is an ROG-cone, these extreme directions come from rank-one elements; we see that $e$ is a finite sum of rank-one elements. Thus, $(\V,p,e)$ admits a scaled Jordan frame. An application of the above theorem shows that $(\V, p^{(m)},e)$ is minimal.
	
	In the case of a Euclidean Jordan algebra, the unit element is a sum of primitive idempotents. Again, we apply the above theorem.
\end{proof}

\begin{example}\label{elementary polynomial}
	On $\Rn$, $n\geq 2$, consider $p(x) \coloneq x_1x_2\cdots x_n$.  Consider the derivative polynomials of $p$, namely, the elementary symmetric polynomials 
	\[ E_{k}(x) = \quad \sum_{\mathclap{i_1 < i_2 < \cdots < i_k}} \quad x_{i_1} x_{i_2} \cdots x_{i_k}. \]   
	Now in the Euclidean Jordan algebra $\Rn$, the standard coordinate vectors are primitive idempotents, hence of rank one. Moreover, their sum is $\one$. By Theorem \ref{scaled JF property for the mth derivative} applied to $(\Rn, E_n,\one)$,  each $E_k$, $k\geq 1$, is minimal. We remark that the minimality of $E_k$ can also be deduced from \cite[Proposition 3.5]{ito2023automorphisms}.
\end{example}

\begin{remark}
	The existence of a scaled Jordan frame is a sufficient condition for the system to be minimal. Example \ref{example3.6} shows that one could have a minimal system without a scaled Jordan frame.
\end{remark}

\begin{example}
	Consider $\R^3$ with $p(x_1,x_2,x_3)=x_1x_2$ and $e=(1,1,0)$. Then $\lplus$ is not proper (hence cannot be an ROG-cone) yet has a (scaled) Jordan frame, namely, $\{e_1,e_2\}$. 
\end{example}

\begin{remark} (The ROG-cone property is not hereditary)
	\emph{If $n \geq 4$ and $\lplus$ is an $\mathrm{ROG}$-cone, then the derivative cone $\lplus^\prime$ can never be an $\mathrm{ROG}$-cone.} We see this as follows. Assume that both $\lplus$ and $\lplus^\prime$ are ROG-cones. Then, as per \cite[Proposition 3.5]{ito2023automorphisms}, $\lplus^\prime$ (which is proper) is strictly larger than $\lplus$. Hence, $\lplus^\prime$ contains an extreme direction that is not in $\lplus$. Now, from \cite[Theorem 12]{renegar2006hyperbolic}, all rank-one elements of $\lplus^\prime$ are already in $\lplus$; thus, this extreme direction can never be of rank one. This means that $\lplus^\prime$ can never be an ROG-cone. On the other hand, when $n = 3$, $\lplus^{\prime}$ may inherit the ROG-cone property from $\lplus$. To see this, consider $(\R^{3}, E_3, \one)$ and Proposition 3.12 along with the remarks at the end of p.\,255 of \cite{ito2023automorphisms}.
\end{remark}

We now describe some examples and results dealing with Jordan frames.

\begin{example} \label{example4.3}
	In a Euclidean Jordan algebra $\V$ of rank $n$ with unit $e$, let $\mathcal{C}$ be any set that contains a Jordan frame $\{e_1, e_2, \ldots, e_n\}$. Consider the hyperbolic system $(\W, q, e)$, where $\W$ is the span of $\mathcal{C}$ and $q(x) = \det(x)$. In  $(\W, q, e)$, $\{e_1, e_2, \ldots, e_n\}$ is a Jordan frame. 
\end{example}

Deviating from Example \ref{example4.3}, we provide the following.

\begin{example}
	Consider $\Rn$, where $n \geq 2$. With $\W$ denoting any real finite-dimensional vector space, let $\V \coloneq \Rn \times \W$. For any $v \in \V$, we write $v = (x, w)$, where $x = (x_1, x_2, \ldots, x_n)^T \in \Rn$ and $w \in \W$. We let $p(v) \coloneq x_1 x_2 \cdots x_n$ and $e = (\one_n, 0)$. Then, $(\V, p, e)$ is a hyperbolic system of degree $n$ with $\lambda(v) = x^\downarrow$. Moreover, any Jordan frame in this system is of the form $\set{(e_i, w_i)}{i = 1, 2, \ldots, n}$, where $\{e_1, e_2, \ldots, e_n\}$ is the standard coordinate system in $\Rn$ and $\sum_{i=1}^{n}w_i = 0$. Note that if $0 \neq w \in \W$, then $(0, w)$ cannot be written as a linear combination of elements in a Jordan frame. 
\end{example}

\begin{example} \label{example4.4}
	On $\R^4$, consider the elementary polynomial of degree $3$: $p(x) = x_1x_2x_3 + x_1x_2x_4 + x_1x_3x_4 + x_2x_3x_4$. With $e=\one$, $(\R^4, p, e)$ becomes a hyperbolic system in which the standard coordinate system $\{e_1, e_2, e_3, e_4\}$ is a scaled Jordan frame but not a Jordan frame (as $\lambda(e_1)=(\frac{3}{4},0,0)$). 
\end{example}

Motivated by the above example, we raise the following question: Is the property of having a Jordan frame hereditary? The proposition below answers this and shows much more: A derivative system cannot have a Jordan frame unless $n \leq 3$.

\begin{proposition}
	Suppose $(\V, p, e)$ has degree $n \geq 4$. Then its derivative system $(\V, p', e)$ cannot have a Jordan frame. 
\end{proposition}

\begin{proof}
	Assume, if possible, $\{c_1, c_2, \ldots, c_{n-1}\}$ is a Jordan frame in $(\V, p^\prime, e)$. Then each $c_i$ is in $\lplus^\prime$ and has rank one in $(\V, p^\prime, e)$.  As $n \geq 4$, Proposition \ref{rank preservation}$(b)$ implies that $c_i\in \lplus$ and has rank one in $(\V, p, e)$. Since $c_1 + c_2 + \cdots + c_{n-1} = e \in \Lambda_{++}$, $\{c_1, c_2, \ldots, c_{n-1}\}$ forms a scaled Jordan frame in $(\V, p, e)$. Then from Theorem \ref{thm on scaled JF}, $n-1 \ge n$, which is not possible. Thus, $(\V, p', e)$ cannot have a Jordan frame.
\end{proof}

\begin{remark}
	On $\Rn$, $n\geq 4$, consider the elementary polynomial $E_k$ (see Example \ref{elementary polynomial}). The above proposition shows that  $(\R^n, E_k, \one)$ cannot have a Jordan frame when $3 \le k \le n-1$. However, this is not true when $n=3$: consider $(\R^3, E_2, \one)$, the derivative system of $(\R^3, E_3, \one)$. It is easy to verify that $\big \{(\tfrac{3}{2}, 0, 0)^T, (-\tfrac{1}{2}, 1, 1)^T\big \}$ is a Jordan frame of $(\R^3, E_2, \one)$.
\end{remark}

We end this section with a characterization of system-automorphisms.

\begin{corollary} \label{system-auto characterization}
	Suppose $(\V,p,e)$ is complete and admits a scaled Jordan frame. Then the following are equivalent for a linear transformation $A: \V \to \V$:
	\begin{itemize}
		\item[$(i)$] $A$ is a system-automorphism.
		\item[$(ii)$] $A$ is a cone-automorphism and $A(e)=e$.
	\end{itemize}
\end{corollary}

\begin{proof}
	$(i) \Rightarrow (ii)$: Assume that $A$ is a system-automorphism. This means that $A$ is invertible and $\lambda(Ax)=\lambda(x)$ for all $x$. As noted in Proposition \ref{prop properties of automorphism}, $A$ is a cone-automorphism. As our system is complete, from $\lambda(Ae)=\lambda(e)$ and (\ref{translation property of lambda}) we get $Ae=e$. 
	
	$(ii) \Rightarrow (i)$: Assume that $(ii)$ holds. As $(\V,p,e)$ admits a scaled Jordan frame, we see from Theorem \ref{scaled JF property for the mth derivative} that $p$ is a minimal polynomial. Then \cite[Proposition 2.6]{ito2023automorphisms} applies: there exists a positive constant $\alpha$ such that 
	\[ p(Ax) = \alpha\, p(x) \quad (\forall\, x \in \V). \]
	In particular, as $Ae=e$, we have 
	\[ p(te-Ax) = p(tAe-Ax) = \alpha\, p(te-x) \quad (\forall\, t \in \R,\; x \in \V). \]
	This shows that $\lambda(Ax)=\lambda(x)$ for all $x\in \V$. As $A$ is invertible (recall $A(\lplus)=\lplus$), we see that $A$ is a system-automorphism.
\end{proof}

\begin{remark}
	In the proof of $(i) \Rightarrow (ii)$, we only used completeness of $p$, while in the proof of $(ii) \Rightarrow (i)$, we only used the minimality of $p$ to invoke a result of Ito and Louren{\c c}o. We note that M. Orlitzky (in a private communication) also uses these two conditions and similar arguments to show the equivalence of $(i)$ and $(ii)$.   
\end{remark}

\subsection{Properties of a Jordan frame}

We show below that every Jordan frame is orthonormal. First,  we cover a preliminary result.

\begin{theorem} \label{lambda of the sum}
	In a hyperbolic system of degree $n$, suppose $\{c_1, c_2, \ldots, c_k\}$ is a set of mutually orthogonal primitive idempotents. Then $k \leq n$ and $c_1 + c_2 + \cdots + c_k$ is an idempotent. Moreover, when $k = n$ and $p$ is complete, $c_1 + c_2 + \cdots + c_n = e$.
\end{theorem}

\begin{proof}
	Theorem \ref{rank additive result}, together with an induction argument, shows that 
	\[ \lambda(c_1+c_2+\cdots+c_k)=\one_k. \] 
	It follows that $c_1+c_2+\cdots+c_k$ is an idempotent and $k\leq n$. Now suppose $k=n$ and $p$ is complete; let $u = c_1 + c_2 + \cdots + c_n$. Then, $\lambda(u)  = \one$. By \eqref{translation property of lambda}, we have $\lambda(u-e) = \lambda(u) - \one = 0$. Since $p$ is complete, $u = e$.
\end{proof}

\begin{theorem} \label{frame eigenvalue}
	\textup{(Orthonormality in a Jordan frame)} 
	Consider a hyperbolic system $(\V, p, e)$ of degree $n$ with $\{ c_1, c_2, \ldots, c_k\} \subseteq \lplus$. If $\{c_1, c_2, \ldots, c_k\}$ is a Jordan frame, then it is orthonormal and $k = n$. The converse holds when the system is complete, and each $c_i$ is a primitive idempotent.
\end{theorem}

\begin{proof} 
	Suppose $\{c_1, c_2, \ldots, c_k\}$ is a Jordan frame in $(\V,p,e)$. Using \eqref{trace}, we see that  $\ip{e}{c_i} = 1$ for all $1 \leq i \leq k$. Also, since $\lambda$ is norm-preserving, $\ip{c_i}{c_i} = \norm{c_i}^2 = \norm{\lambda(c_i)}^2 = 1$. Then, $e = c_1 + c_2 + \cdots + c_k$ implies that
	\[ 1 = \ip{e}{c_i} = \norm{c_i}^2 + \sum_{j \neq i} \ip{c_j}{c_i} = 1 + \sum_{j \neq i} \ip{c_j}{c_i}. \]
	However, $c_i, c_j \in \lplus \subseteq (\lplus)^*$ implies that $\ip{c_i}{c_j} \geq 0$. Thus, we have $\ip{c_i}{c_j} = 0$ for $i \neq j$. As $\ip{c_i}{c_i} = 1$ for all $i$, we see that the set $\{c_1,c_2,\ldots, c_k\}$ is orthonormal; in particular, it is linearly independent. Moreover, 
	\[ n = \norm{e}^2 = \norm{c_1 + c_2 + \cdots + c_k}^2 = \sum_{i=1}^{k} \norm{c_i}^2  = k. \]
	(This equality can also be seen from Theorem \ref{thm on scaled JF}.)
	Thus, we have proved the first part of the theorem.
	
	For the second part, assume that $\{c_1, c_2, \ldots, c_n\}$ is a set of mutually orthogonal primitive idempotents and $p$ is complete. By Theorem \ref{lambda of the sum}, $c_1 + c_2 + \cdots + c_n = e$. So, $\{c_1, c_2, \ldots, c_n\}$ is a Jordan frame.
\end{proof}

The above theorem allows us to say that 
\begin{center}
	\begin{minipage}{0.85\textwidth}
		\emph{A Jordan frame in a hyperbolic system $(\V,p,e)$ of degree $n$ is an orthonormal set consisting of $n$ primitive idempotents with sum $e$.}
	\end{minipage}
\end{center}

We note that as a consequence of the above, a Jordan frame in $(\V, p, e)$ cannot be a Jordan frame in $(\V, p^\prime, e)$.

\gap
The following result specifies conditions for a set of rank-one elements to be a Jordan frame.

\begin{theorem}
	Let $(\V, p, e)$ be a hyperbolic system of degree $n$ and $\{c_1, c_2, \ldots, c_k\}$ be a set of rank-one elements in $\lplus$. Consider the following statements:
	\begin{itemize}
		\item[$(i)$] Each $c_i$ is a primitive idempotent.
		\item[$(ii)$] $c_1 + c_2 + \cdots + c_k = e$.
		\item[$(iii)$] $\ip{c_i}{c_j} = 0$ for all $i \neq j$.
	\end{itemize}
	Then $(i) + (ii) \Rightarrow (iii)$ and $(ii) + (iii) \Rightarrow (i)$. Also, $(i) + (iii) \Rightarrow (ii)$ when $k = n$ and $p$ is complete.
\end{theorem}

\begin{proof}
	When $(i)$ and $(ii)$ hold, we have a Jordan frame. The implication $(i)+(ii)\Rightarrow (iii)$ follows from Theorem \ref{frame eigenvalue}.
	
	When $(ii)$ and $(iii)$ hold, we proceed as in the proof of Theorem 4.2 with $d=e$. Then (\ref{alphai equality}) shows that $\alpha_i = 1$ for all $i$, where $\alpha_i$ is the nonzero eigenvalue of $c_i$. Thus, $(i)$ holds.
	
	When $k=n$ and $p$ is complete, the implication $(i) + (iii) \Rightarrow (ii)$ comes from Theorem \ref{frame eigenvalue}. 
\end{proof}

\begin{theorem} \label{basic theorem on frame} 
	In the hyperbolic system $(\V,p,e)$ of degree $n$, let $\{c_1, c_2, \ldots, c_n\}$ be a Jordan frame. For real numbers $r_1, r_2, \ldots, r_n$, define
	\[ x = r_1c_1 + r_2c_2 + \cdots + r_nc_n \quad \text{and} \quad r = (r_1, r_2, \ldots, r_n)^T. \] 
	Then, $\lambda(x) = r^\downarrow$.
\end{theorem}

\begin{proof}
	As $e = c_1 + c_2 + \cdots + c_n$, we have, for any $t \in \R$, $x + te =(r_1+t)c_1 + (r_2+t)c_2 + \cdots + (r_n+t)c_n$. Since $\lambda(x+te) = \lambda(x) + t\one$, showing $\lambda(x)=r^\downarrow$ is the same as showing $\lambda(x+te)=(r+t\one)^\downarrow.$ Hence, by adding a suitable multiple of $e$ to $x$, we may assume that $r_i > 0$ for all $i = 1, 2, \ldots, n$. We also assume without loss of generality that $r_1 \geq r_2 \geq \cdots \geq r_n > 0$.
	
	With these in place, we show by induction that for all $1 \leq k \leq n$, 
	\begin{equation} \label{induction}
		\lambda(r_1c_1 + r_2c_2 + \cdots + r_kc_k) = (r_1, r_2, \ldots, r_k, 0, \ldots, 0)^T.
	\end{equation}
	Then, putting $k=n$, we get our stated result.
	
	The above statement holds when $k=1$ as $\lambda$ is positively homogeneous and $\lambda(c_1) = (1, 0, 0, \ldots, 0)^T$. For the induction step, assume (\ref{induction}) holds for some $k$, $1 \leq k < n$. We show 
	\begin{equation} \label{induction statement for k+1}
		\lambda(r_1c_1 + r_2c_2 + \cdots + r_kc_k + r_{k+1}c_{k+1}) = (r_1, r_2, \ldots, r_k, r_{k+1}, 0, \ldots, 0)^T. 
	\end{equation}
	Let $a = r_1c_1 + r_2c_2 + \cdots + r_kc_k$  and $b = r_{k+1}c_{k+1}$. Then  by the induction hypothesis,
	\[ \lambda(a) = (r_1, r_2, \ldots, r_k, 0, \ldots, 0)^T \quad \text{and} \quad \lambda(b)=(r_{k+1},0,\ldots, 0)^T. \]  
	Since $a, b \in \lplus$ with $\lambda(a) \neq 0 \neq \lambda(b)$ and $\ip{a}{b} = 0$, from Theorem \ref{rank additive result}
	\[ \lambda(a+b) = \big[ (r_1, r_2, \ldots, r_k, s_{k+1}, \ldots, s_n)^T \big]^\downarrow, \]
	where each $s_i$ is an eigenvalue of $b$ and all the nonzero eigenvalues of $b$ are contained in $\{s_{k+1}, s_{k+2}, \ldots, s_n\}$. As $r_{k+1}$ is the only nonzero eigenvalue of $b$ and  $r_1 \geq r_2 \geq \cdots \geq r_{k} \geq r_{k+1}$, we have
	\[ \lambda(a+b) = (r_1, r_2, \ldots, r_k, r_{k+1}, 0, \ldots, 0)^T. \]
	This gives \eqref{induction statement for k+1}. Our assertion then follows from \eqref{induction} by putting $k=n$.
\end{proof}

We now present our key result, which addresses the question of when a hyperbolic system of degree $n$ admits a Jordan frame.

\begin{theorem}
	A hyperbolic system $(\V, p, e)$ of degree $n$ admits a Jordan frame if and only if the system contains a subsystem $(\W,p,e)$ that is isomorphic to  $(\Rn, E_n,\one)$ $($equivalently, $(\V, p, e)$ contains a copy of the Euclidean Jordan algebra $\Rn)$.
\end{theorem}

Here, we say that $(\W,p,e)$ is a \emph{subsystem} of $(\V,p,e)$ to mean that $\W$ is a (linear) subspace of $\V$ that contains $e$ (in which case, $p$ and $\lambda$ are restricted to $\W$ to make $(\W,p,e)$ into a hyperbolic system). Also, an \emph{isomorphism} between $(\W,p,e)$ and $(\Rn,E_n,\one)$ is a bijective linear transformation $\phi:\W\rightarrow \Rn$ such that $\phi(x)^\downarrow =\lambda(x)$ for all $x\in \W$.

\begin{proof}
	Suppose $(\V,p,e)$ has a Jordan frame $\mathcal{F} \coloneq \{c_1, c_2, \ldots, c_n\}$. Let $\W = \operatorname{span}(\mathcal{F})$ and $x,y\in \W$. By the orthonormality of $\mathcal{F}$, we have unique representations $x = \sum_{i=1}^{n} x_ic_i$ and $y = \sum_{i=1}^{n} y_ic_i$. By Theorem \ref{basic theorem on frame}, $\lambda(x) = \big[ (x_1, x_2, \dots, x_n)^T \big]^\downarrow$; moreover, from \eqref{ip}, $$\ip{x}{y} = \sum_{i=1}^{n} x_iy_i.$$ Now the transformation $\phi : \W \to \Rn$ defined by $\phi(x) = (x_1,x_2,\ldots,x_n)^T$ is an isomorphism between the hyperbolic systems $(\W,p,e)$ and $(\Rn,E_n,\one)$. 
	To see that  $(\W,p,e)$ can be turned into a Euclidean Jordan algebra, we define the product
	\[ x \circ y \coloneq \sum_{i=1}^{n} x_iy_ic_i. \]
	With these, we easily verify that $(\W, \ip{\cdot}{\cdot}, \circ)$ is a Euclidean Jordan algebra with unit $e$. In this algebra, $\mathcal{F}$ is a Jordan frame and 
	$\det(x) = x_1x_2 \cdots x_n$, which is the product of eigenvalues of $x$. So this Euclidean Jordan algebra is an isomorphic copy of the Euclidean Jordan algebra $\Rn$.
	
	Conversely, suppose there exist a linear subspace $\W$ of $\V$ containing $e$ and a linear isomorphism $\phi : \W \to \Rn$ such that $\phi(x)^{\downarrow} = \lambda(x)$ for all $x \in \W$. Then $\V$ has a Jordan frame, namely, $\{\phi^{-1}(e_1),\phi^{-1}(e_2),\ldots, \phi^{-1}(e_n)\}$, where $\{e_1,e_2,\ldots, e_n\}$ is the usual coordinate system in $\Rn$. We note that this Jordan frame induces a Euclidean Jordan algebra structure on $\W$ that is isomorphic to the algebra $\Rn$. This completes the proof of the theorem.
\end{proof}

\section{Doubly stochastic transformations and majorization}

We recall from Definition \ref{definition: ds} that a linear transformation on a hyperbolic system $(\V, p, e)$ is doubly stochastic if it is positive, unital, and trace-preserving. When the system is complete, every system-automorphism $T$ is doubly stochastic. Another simple example is the transformation $T$ defined by $T(x)=\frac{\ip{x}{e}}{n}\,e$. Note that in both cases, $\lambda(T(x))\prec \lambda(x)$ for all $x \in \V$. 
The goal of this section is to describe a majorization result for a certain type of doubly stochastic transformation. 

\gap
In the following, unless mentioned otherwise, we assume that $(\V,p,e)$ is a hyperbolic system of degree $n$. We start with the following concept introduced in \cite[Definition 2.4]{gurvits2004combinatorics}.

\begin{definition}
	In a hyperbolic system $(\V, p, e)$ of degree $n$, consider an $n$-tuple of objects written in the (matrix) form: $\bfA \coloneq \big[ a_1, a_2, \ldots, a_n \big]$. We say that $\bfA$ is \emph{$e$-doubly stochastic} if $a_i \in \lplus$ and $\tr(a_i) = 1$ for all $1 \leq i \leq n$, with $a_1 + a_2 + \cdots + a_n = e$.
\end{definition}

Here are some examples.

\begin{example} \label{DS example 1}
	If $\{c_1, c_2, \ldots, c_n\}$ is a Jordan frame, then $\bfA = [c_1, c_2, \ldots, c_n]$ is $e$-doubly stochastic. The converse need not be true: Take $\bfA = \big[ a_1, a_2, \ldots, a_n \big]$, where $a_i = \frac{1}{n}e$ for all $i$. 
\end{example}

\begin{example} \label{DS example 2}
	Let $T$ be a doubly stochastic transformation and $\bfA=\big[ a_1, a_2, \ldots, a_n \big]$ be $e$-doubly stochastic. Then  $\bfB = \big[ T(a_1), T(a_2),\ldots, T(a_n)\big]$ is $e$-doubly stochastic. In particular, this holds when $p$ is complete, and $T$ is a system-automorphism of $\V$.
\end{example}

\begin{example} \label{DS example 3}
	Suppose $\bfA = \big[ a_1, a_2, \ldots, a_n \big]$ is $e$-doubly stochastic.  Let $D = [d_{ij}]$ be an $n\times n$ doubly stochastic matrix. Define $b_i \coloneq \sum_{j=1}^{n} d_{ij}a_j$ for $1 \leq i \leq n$. Then $\mathbf{B} = \big[ b_1, b_2, \ldots, b_n \big]$ is $e$-doubly stochastic.
\end{example}

We now describe some properties of $e$-doubly stochastic $n$-tuples.

\begin{proposition} 
	Suppose $\bfA = \big[ a_1, a_2, \ldots, a_n \big]$ is $e$-doubly stochastic. Then the following hold:
	\begin{itemize}
		\item[$(i)$] The Gram-matrix $G \coloneq \big[ \ip{a_i}{a_j} \big]$ is a doubly stochastic matrix.
		\item[$(ii)$] The matrix 
		$ M \coloneq \big[ \lambda(a_1), \lambda(a_2), \ldots, \lambda(a_n) \big]$
		is column-stochastic. Moreover, $M$ is row-stochastic if and only if $M = \frac{1}{n} \one\one^T$;\\
		when the system is complete, $M$ is doubly stochastic if and only if $a_i = \frac{1}{n}e$ for all $i$.
		\item[$(iii)$] There exists a doubly stochastic matrix $D$ such that the matrix $DM$ is doubly stochastic.
		\item[$(iv)$] When $a_i$'s  are mutually orthogonal, $\{a_1, a_2, \ldots, a_n\}$ is a Jordan frame. 
	\end{itemize}
\end{proposition}

\begin{proof}
	Item $(i)$ is easy to verify.
	
	$(ii)$ The column stochasticity and the `if' part in row-stochasticity of $M$ are easy to check. To see the `only if' part, suppose $M$ is row-stochastic; write $M = \big[ m_{ij} \big]$. Since each column of $M$ has decreasing entries with sum one, for all $i,j$, we have $m_{ij} \geq m_{nj} \geq \frac{1}{n}$. Then, as each row sum in $M$ is one, we see the equality $m_{ij} = \frac{1}{n}$. This shows that $M=\frac{1}{n}\one\one^T$. When $p$ is complete, for any $i$, $\lambda(a_i)=\frac{1}{n}\one$ implies that $a_i=\frac{1}{n}\,e$.
	
	$(iii)$ Since $\bfA$ is $e$-doubly stochastic, by Proposition \ref{gurvits result2},
	\[ \one = \lambda(e) = \lambda(a_1 + a_2 + \cdots + a_n) \prec \lambda(a_1) + \lambda(a_2) + \cdots + \lambda(a_n). \]
	Then by a theorem of Hardy-Littlewood-P{\'o}lya \cite[Theorem 2.B.2]{marshall2011inequalities}, there exists a doubly stochastic matrix $D=[d_{ij}]$ such that 
	\begin{equation} \label{HLP applied}
		\one = D (\lambda(a_1) + \lambda(a_2) + \cdots + \lambda(a_n)).
	\end{equation}
	Clearly, $DM$ is column stochastic. Since each $\lambda(a_i)$ has trace one and $D$ is doubly stochastic, (\ref{HLP applied}) shows that $DM$ is row stochastic. Thus, $DM$ is doubly-stochastic.
	
	$(iv)$ Suppose $a_i$'s are mutually orthogonal. Applying Theorem \ref{rank additive result}, we get
	\[ n =\sum_{i=1}^{n} 1\leq \sum_{i=1}^{n} \rank(a_i) = \rank \bigg( \sum_{i=1}^{n} a_i \bigg) = \rank(e) = n. \] 
	This can happen if and only if $\rank(a_i) = 1$ for all $i$. This, together with $\tr(a_i)=1$ implies that $a_i$'s are primitive idempotents with sum $e$.
\end{proof}

\begin{proposition} \label{e-ds implies lambda-ds} 
	If $\bfA= \big[ a_1, a_2, \ldots, a_n \big]$ is $e$-doubly stochastic, then for every $1 \leq k \leq n$ and every permutation $\sigma$ of $\{1, 2, \ldots, n\}$,
	\begin{equation} \label{eq:lambda-DS}
		\lambda \big( a_{\sigma(1)} + a_{\sigma(2)} + \dots + a_{\sigma(k)} \big) \prec \one_k.
	\end{equation}
	The converse holds when $p$ is complete.
\end{proposition}

\begin{proof}
	Suppose $\bfA$ is $e$-doubly stochastic. Then, for every permutation $\sigma$, $\bfB \coloneq \big[ a_{\sigma(1)}, a_{\sigma(2)}, \dots, a_{\sigma(k)} \big]$ is $e$-doubly stochastic. Hence, it is enough to prove (\ref{eq:lambda-DS}) for the identity permutation. For fixed $1 \leq k \leq n$, define $x = a_{1} + a_{2} + \cdots + a_{k}$. Since $x, e \in \lplus$ and $x \leq e$ (because $a_i \in \lplus$ and $\sum_{i=1}^{n} a_i = e$), we have $\lambda(x) \leq \lambda(e)=\one$ from Proposition \ref{gurvits result2}$(c)$. This implies that $\sum_{j=1}^{m} \lambda_j(x) \leq m$ for all $1 \leq m \leq n$ (and hence for $1 \leq m \leq k-1$.) Moreover, for $k \leq m \leq n$, the linearity of the trace implies 
	\[ \sum_{j=1}^{m} \lambda_j(x) \leq \sum_{j=1}^{n} \lambda_j(x) = \tr(x) = \sum_{j=1}^{k} \tr(a_{j}) = k. \]
	Further, we have equality throughout when $m = n$. This implies that $\lambda(x) \prec \one_k$.
	
	We now prove the reverse implication when $p$ is complete. Assume that $\bfA$ satisfies the condition \eqref{eq:lambda-DS} for all $1 \leq k \leq n$ and a permutation $\sigma$. Then, for every $1\leq i\leq n$, we can choose $\sigma$ so that $i=\sigma(1)$. Consequently, $\lambda(a_i) \prec (1, 0, 0, \ldots, 0)^T$ and so $\tr(a_i)=1$. Moreover, $\lambda(a_i)$ is in the convex hull of the standard basis vectors in $\Rn$, see e.g., \cite[Theorem II.1.10]{bhatia1997matrix}. Thus,  $\lambda(a_i)$ is a nonnegative vector, that is, $a_i \in \lplus$. Since $\lambda(a_1 + a_2 + \cdots + a_n) \prec (1, 1, \ldots, 1)^T = \one$, we have 
	$\lambda(a_1 + a_2 + \cdots + a_n) = \one = \lambda(e)$. By completeness, $a_1 + a_2 + \cdots + a_n = e$. Thus, $\bfA$ is $e$-doubly stochastic.
\end{proof}

\begin{theorem} \label{T_A is DS}
	In a hyperbolic system of degree $n$, suppose $\{c_1, c_2, \ldots, c_n\}$ is a Jordan frame, $\bfA = \big[ a_1, a_2, \ldots, a_n \big]$ is $e$-doubly stochastic, and $D = [d_{ij}]$ is a doubly stochastic matrix. Consider the linear transformation $T : \V \to \V$ given by
	\begin{equation} \label{DS transformation}
		T(x) \coloneq \sum_{i=1}^{n} \sum_{j=1}^{n} d_{ij} \ip{x}{a_j} c_i \quad (\forall\, x \in \V).
	\end{equation}
	Then $T$ is doubly stochastic on $(\V, p, e)$ and $\lambda(T(x)) \prec \lambda(x)$ for all $x \in \V$.
\end{theorem}

\begin{proof} 
	Define $\mathbf{B} = \big[ b_1, b_2, \ldots, b_n \big]$, where $b_i = \sum_{j=1}^{n} d_{ij} a_j$ for all $i$. Then $\mathbf{B}$ is $e$-doubly stochastic (see Example \ref{DS example 3}) and (\ref{DS transformation}) simplifies to
	\[ T(x) = \sum_{i=1}^{n} \ip{x}{b_i} c_i \quad (\forall\, x \in \V). \]
	From Theorem \ref{SG theorem}, we see that $T(\lplus)\subseteq \lplus$. The other two conditions in Definition \ref{definition: ds} are easy to verify.  We now prove the majorization inequality. Fix $x \in \V$, let $r_i \coloneq \ip{x}{b_i}$ and $r \coloneq (r_1, r_2, \ldots, r_n)^T$.
	Then, $T(x) = \sum_{i=1}^{n} r_ic_i$ and by Theorem \ref{basic theorem on frame}, $\lambda(T(x)) = r^\downarrow$. We now claim that $r \prec \lambda(x)$.
	Assume, without loss of generality, that $r_1 \geq r_2 \geq \cdots \geq r_n$ (which can be achieved by rearranging $b_i$s in $\mathbf{B}$). Then, for any $k$, $1 \leq k \leq n$, 
	\begin{align*}
		r_1 + r_2 + \cdots + r_k & = \ip{x}{b_1 + b_2 + \cdots + b_k} \\
		& \leq \ip{\lambda(x)}{\lambda(b_1 + b_2 + \cdots + b_k)} \\
		& \leq \ip{\lambda(x)}{\one_k},
	\end{align*}
	where the first inequality is due to the inner-product-expanding property of $\lambda$; the second inequality holds due to Proposition \ref{e-ds implies lambda-ds} and the fact that if $u \prec v$ in $\Rn$ and $w \in \Rn$, then $\ip{w^{\downarrow}}{u^{\downarrow}} \leq \ip{w^{\downarrow}}{v^{\downarrow}}$; see \cite[Problem II.5.16]{bhatia1997matrix}.
	Moreover, for $k=n$,
	\[ \sum_{i=1}^{n} r_i = \bigg\langle x, \sum_{i=1}^{n} b_i \bigg\rangle = \ip{x}{e} = \tr(x) = \sum_{i=1}^{n} \lambda_i(x), \]
	where the second equality is due to $\mathbf{B}$ being $e$-doubly stochastic. Thus, $r \prec \lambda(x)$. Hence, $\lambda(T(x)) = r^\downarrow \prec \lambda(x)$.
\end{proof}

\begin{remark}
	Suppose $\bfA = \big[ a_1, a_2, \ldots, a_n \big]$ and $\bfB = \big[ b_1, b_2, \ldots, b_n \big]$ are $e$-doubly stochastic $n$-tuples. Let $T: \V \to \V$ be a doubly stochastic linear transformation. Then, it is straightforward to show that 
	the matrix $D \coloneq \big[\! \ip{a_i}{T(b_j)} \!\big]$ is doubly stochastic. In particular, if $\{c_1, c_2, \ldots, c_n\}$ is a Jordan frame, then the matrix $\big[\! \ip{c_i}{T(c_j)} \!\big]$ is doubly stochastic.
\end{remark}

\begin{example}
	Let $\{c_1, c_2, \ldots, c_n\}$ be a Jordan frame in a hyperbolic system of degree $n$ and $D \coloneq [d_{ij}]$ be an $n\times n$ real matrix. Consider the linear transformation   $T : \V \to \V$ given by
	\[ T(x) = \sum_{i=1}^{n} \sum_{j=1}^{n} d_{ij} \ip{x}{c_j} c_i. \]
	Since $\{c_1, c_2, \ldots, c_n\}$ is an orthonormal set, we have $D = \big[\! \ip{c_i}{T(c_j)} \!\big]$. Now Theorem \ref{T_A is DS}, together with the above remark, shows that $D$ is doubly stochastic if and only if $T$ is doubly stochastic. 
\end{example}

Finally, by specializing Theorem \ref{T_A is DS}, we state an analog of Schur's theorem in the setting of a hyperbolic system $(\V, p, e)$.

\begin{example} \label{Schur theorem}
	Consider a Jordan frame $\{c_1, c_2, \ldots, c_n\}$ in $\V$. Then the linear transformation 
	\begin{equation} \label{schur transformation} 
		\diag(x) \coloneq \ip{x}{c_1} c_1 + \ip{x}{c_2} c_2 + \cdots + \ip{x}{c_n} c_n\quad (x\in \V) 
	\end{equation}
	is doubly stochastic and $\lambda(\diag(x)) \prec \lambda(x)$.
\end{example}

Motivated by a result in the setting of Euclidean Jordan algebras, see the Introduction, we pose the following 

\gap
\noindent \textbf{Problem:} Given a Jordan frame $\{c_1, c_2, \ldots, c_n\}$ and an $e$-doubly stochastic $n$-tuple $\bfA = \big[ a_1, a_2, \ldots, a_n \big]$, we know that the transformation $T$ defined by $T(x)=\sum_{i=1}^{n} \ip{x}{a_i} c_i$ is doubly stochastic and $\lambda(T(x))\prec \lambda(x)$.  Now consider the transformation $S : \V \to \V$ given by 
\[ S(x) \coloneq \sum_{i=1}^{n} \ip{x}{c_i} a_i \quad (x \in \V). \]
As $\ip{T(x)}{y} = \ip{x}{S(y)}$, we can think of $S$ as the `adjoint' of $T$ relative to the (semi-)inner product in our hyperbolic system. It is easy to show that $S$ is doubly stochastic. Can we say that $\lambda (S(x)) \prec \lambda(x)$ for all $x \in \V$?

\gap
We end this paper with the following remark/result on the characterization of a system-automorphism expressed in terms of doubly stochastic transformations. This can be seen via an application of Corollary \ref{system-auto characterization}. 

\begin{remark} \label{system-auto characterization 2}
	Suppose $(\V,p,e)$ is complete and admits a scaled Jordan frame (or, more generally, $p$ is a minimal polynomial). Then the following are equivalent for a linear transformation $A: \V \to \V$:
	\begin{itemize}
		\item[$(i)$] $A$ is a system-automorphism.
		\item[$(ii)$] $A$ is a doubly stochastic cone-automorphism.
		\item[$(iii)$] $A$ is invertible and both $A$ and $A^{-1}$ are doubly stochastic.
	\end{itemize}
\end{remark}

\section*{Acknowledgment}

The work of Juyoung Jeong was supported by the National Research Foundation of Korea (NRF) grant funded by the Ministry of Science and ICT (No.\,RS-2026-25475143), and by the Global-LAMP Program of the NRF grant funded by the Ministry of Education (No.\,RS-2025-25441317).


    
\end{document}